\documentclass[12pt,american]{article}
\usepackage[T1]{fontenc}
\usepackage[latin9]{inputenc}
\usepackage{babel}
\usepackage{amstext}
\usepackage{amsthm}
\usepackage{amssymb}
\usepackage[unicode=true,
 bookmarks=true,bookmarksnumbered=true,bookmarksopen=false,
 breaklinks=true,pdfborder={0 0 1},backref=false,colorlinks=true]
 {hyperref}
\hypersetup{pdftitle={Green Measures for Markov Processes},
 pdfauthor={Yu. G. Kondratiev and J. L. da silva},
 pdfsubject={Green measures},
 pdfkeywords={Green measures, Markov processes, additive functionals},
 linkcolor=black,citecolor=black,urlcolor=black,filecolor=black}

\makeatletter
\theoremstyle{plain}
\newtheorem{thm}{\protect\theoremname}[section]
\theoremstyle{definition}
\newtheorem{defn}[thm]{\protect\definitionname}
\theoremstyle{plain}
\newtheorem{prop}[thm]{\protect\propositionname}
\theoremstyle{remark}
\newtheorem{rem}[thm]{\protect\remarkname}

\usepackage{babel}
\usepackage{mathrsfs}
\usepackage[intlimits]{amsmath}

\usepackage{babel}
\usepackage{mathrsfs}
\usepackage{bbm}

\numberwithin{equation}{section}
\newcommand\R{{\mathbb R}}
\newcommand\X{{\R^d}}
\renewcommand{\G}{\mathcal G}
\newcommand\la{\lambda}
\newcommand\B{{\mathcal B}}
\renewcommand\L{{\mathcal L}}

\makeatother

\providecommand{\definitionname}{Definition}
\providecommand{\propositionname}{Proposition}
\providecommand{\remarkname}{Remark}
\providecommand{\theoremname}{Theorem}

\begin{document}
\title{Green Measures for Markov Processes}
\author{\textbf{Yuri Kondratiev}\\
 Department of Mathematics, University of Bielefeld, \\
 D-33615 Bielefeld, Germany,\\
 Dragomanov University, Kiev, Ukraine\\
 Email: kondrat@mathematik.uni-bielefeld.de\\
 Email: kondrat@math.uni-bielefeld.de\and\textbf{ Jos{\'e} L.~da
Silva},\\
 CIMA, University of Madeira, Campus da Penteada,\\
 9020-105 Funchal, Portugal.\\
 Email: joses@staff.uma.pt}
\date{\today}

\maketitle
 
\begin{abstract}
In this paper we study Green measures of certain classes of Markov
processes. In particular Brownian motion and processes with jump generators
with different tails. The Green measures are represented as a sum
of a singular and a regular part given in terms of the jump generator.
The main technical question is to find a bound for the regular
part.


\end{abstract}
\tableofcontents{}

\section{Introduction}

Let $X(t), t\geq 0$ be a time homogeneous Markov process in $\X$ starting from the point $x\in\X$. 
For a function $f:\X\to \R$ we consider the following heuristic object
$$
V(f,x)= \int_0^\infty E^x[f(X(t)] dt.
$$
If this quantity exists then $V(f,x)$ is called the potential for the function $f$.
The notions of potentials is well known in probability theory, see e.g., \cite{Blumenthal1968,Revuz-Yor-94}.

The existence of the potential $V(f,x) $ is a difficult question and the class
of admissible $f$ shall be analyzed for each process separately. 
An alternative approach is based on the use of the generator $L$
of the process $X$.  Namely, the potential $V(f,x) $ may be constructed
as the solution to the following equation:
$$
-LV=f.
$$
Of course, there appear a technical problem of the characterization of
the domain for the inverse generator $L^{-1}$.  For a general Markov process
we can not characterize this domain. 

 In the analogy with
the PDE framework, we would like to have a representation
$$
V(t,x)  =\int_{\X} f(y) \G(x,\mathrm{d}y),
$$
where $\G(x,dy)$ is a measure on $\X$.  This measure is nothing but the fundamental
solution to the considered equation and traditionally may be called the Green measure
for the operator $L$. Our aim is to study Green measures for certain classes of Markov processes.
We stress again  that the concepts of potentials and Green measures are related
in the same manner as the notions of solutions and fundamental solutions
to the corresponding equations.

In this paper we discuss general notion of Green measures for Markov processes and
consider particular examples of such processes for which the existence and properties of 
Green measures may be analyzed.

\section{General Framework}

We will consider time homogeneous Markov processes $X(t)$, $t\ge0$
in $\X$, see for example \cite{Blumenthal1968,Dynkin1965, Meyer1967, Revuz-Yor-94}. A standard way to define a Markov process is to
give the probability $P_{t}(x,B)$ of the transition from the point
$x\in\X$ to the Borel set $B\subset\X$ in the time $t>0$.
In some cases we may have 
\[
P_{t}(x,B)=\int_{B}p_{t}(x,y)\,\mathrm{d}y,
\]
where $p_{t}(x,y)$ is the density of the transition probability (heat
kernel), that is, $P_{t}(x,\mathrm{d}y)=p_{t}(x,y)\,\mathrm{d}y$.
The function 
\[
g(x,y):=\int_{0}^{\infty}p_{t}(x,y)\,\mathrm{d}t
\]
is called the Green function, although the integral here may diverge.
The existence of the Green function for a given process or transition
probability, even for simple classes of Markov processes, is not always
guaranteed, see examples below. Nevertheless, Green functions for
different classes of Markov processes are well known in probability
theory, see, e.g., \cite{CGL20,Grigoryan2018a} and references therein.

As an alternative we introduce the Green measure by 
\[
\G(x,\mathrm{d}y):=\int_{0}^{\infty}P_{t}(x,\mathrm{d}y)\,\mathrm{d}t,\quad x\in\X,
\]
assuming the existence of this object as a Radon measure on $\X$.
Then we would have $\G(x,\mathrm{d}y)=g(x,y)\,\mathrm{d}y$. The aim
of this paper is to show how to define and study Green measures for
certain particular Markov processes in $\X$.

Let us be more precise. We can start with a Markov semigroup $T(t),t\geq0$,
that is, a family of linear operators in a Banach space $E$. As $E$,
we may use bounded measurable functions $B(\X)$, bounded continuous
functions $C_{b}(\X)$ or Lebesgue spaces $L^{p}(\X)$, $p\ge1$ depending
on each particular case. This family of operators satisfy the following
properties:
\begin{enumerate}
\item[(i)] $T(t)\in\L(E),\;\;t\geq0$,
\item[(ii)] $T(0)=1$,
\item[(iii)] ${\displaystyle \lim_{t\to0^{+}}}T(t)f=f,\quad f\in E$,
\item[(iv)] $T(t+s)=T(t)T(s)$,
\item[(v)] $\forall f\geq0\;\;\;T(t)f\geq0$.
\end{enumerate}
The semigroup is conservative if
\begin{enumerate}
\item[(vi)] $T(t)1=1$.
\end{enumerate}
The semigroup $T(t)$, $t\ge0$ is associated with a Markov process
$\{X(t),t\geq0\;|\;P_{x},x\in\X\}$ if

\[
(T(t)f)(x)=E^{x}[f(X(t))]=\int_{\X}f(y)P_{t}(x,\mathrm{d}y),\quad f\in E.
\]
The transition probabilities may be constructed from the semigroup
by choosing $f=\mathbbm{1}_{A}$, $A\in\B(\X)$, that is,
\[
P_{t}(x,A)=(T(t)\mathbbm{1}_{A})(x).
\]
Now we introduce the resolvent of the Markov semigroup $T(t)$, $t\ge0$.
Let $\lambda>0$ be given. The $\lambda$-resolvent of the Markov
semigroup is the linear operator $R_{\lambda}:E\longrightarrow E$
defined by
\[
(R_{\lambda}f)(x):=\int_{0}^{\infty}e^{-\lambda t}(T(t)f)(x)\,\mathrm{d}t=\int_{0}^{\infty}e^{-\lambda t}\int_{\X}f(y)P_{t}(x,\mathrm{d}y)\,\mathrm{d}t,
\]
for any $f\in E$ and $x\in\X$. 

Denote by $\B_{b}(\X)$ the family of bounded Borel sets in $\X$
and $C_{0}(\X)$ the space of continuous functions with compact support.
\begin{defn}
The Green measure for a Markov process $X(t)$, $t\ge0$ with transition
probability $P_{t}(x,B)$ is defined by
\[
\G(x,B):=\int_{0}^{\infty}P_{t}(x,B)\,\mathrm{d}t,\;\;B\in\B_{b}(\X),
\]
or 
\[
\int_{\X}f(y)\G(x,dy)=\int_{0}^{\infty}f(y)P_{t}(x,dy)\,\mathrm{d}t,\;\;f\in C_{0}(\X)
\]
whenever these integrals exist.
\end{defn}

The Markov generator is a characteristic of a Markov semigroup. More
precisely, we have the following definition.
\begin{defn}
We set 
\[
D(L):=\left\{ f\in E\,\middle|\,\frac{T(t)f-f}{t}\;\mathrm{converges\;in}\;E\;\mathrm{when}\;t\to0^{+}\right\} 
\]
and for every $f\in D(L)$, $x\in\X$
\[
(Lf)(x):=\lim_{t\to0^{+}}\frac{(T(t)f)(x)-f(x)}{t}.
\]
Then $D(L)$ (the domain of $L$) is a linear subspace of $E$ and
${L:D(L)\longrightarrow E}$ is a linear operator called the generator
of the semigroup $T(t)$, $t\ge0$.

There are several known properties of the generator $L$ and and
we have the full description of the Markov generators via so-called
maximum principle, see, e.g., \cite{CGL20,Grigoryan2018a}.
\end{defn}

From the relation between semigroup and generator we have
\begin{equation}
E^{x}\left[\int_{0}^{\infty}f(X(t))\,\mathrm{d}t\right]=\int_{\X}f(y)\G(x,\mathrm{d}y)=-(L^{-1}f)(x)=\int_{0}^{\infty}(T(t)f)(x)\,\mathrm{d}t\label{FS}
\end{equation}
 for every $f\in C_{0}(\X)$. Because any Radon measure defines a
generalized function, then we may write 
\[
\G(x,\mathrm{d}y)=g(x,y)\,\mathrm{d}y,
\]
where $g(x,\cdot)\in D'(\X)$ is a positive generalized function for
all $x\in\X$. In view of (\ref{FS}) the Green measure is the fundamental
solution corresponding to the operator $L$. Note that the existence
and regularity of this fundamental solution produces a description
of admissible Markov processes for which the Green measure exists.
In Section \ref{sec:Particular-Models} we present some examples and
show the existence of the Green measure under the assumption $d\geq3$.
This moment is one more demonstration on the essential influence of
the dimension of the phase space on the properties of Markov processes.

\section{Jump Generators and Green Measures}

Let $a:\X\to\R$ be a fixed kernel with the following properties: 
\begin{enumerate}
\item Symmetric, $a(-x)=a(x)$, for every $x\in\X$.
\item Positive and bounded, $a\geq0$, $a\in C_{b}(\X)$. 
\item Integrable
\[
\int_{\X}a(y)\,\mathrm{d}y=1.
\]
\end{enumerate}
Consider the generator $L$ defined on $E$ (as mentioned above) by
\[
(Lf)(x)=\int_{\X}a(x-y)[f(y)-f(x)]\,\mathrm{d}y=(a*f)(x)-f(x),\quad x\in\X.
\]
In particular, $L^{\ast}=L$ in $L^{2}(\X)$ and $L$ is a bounded
linear operator in all $L^{p}(\X)$, $p\ge1$. We call this operator
the jump generator with jump kernel $a$. The corresponding Markov
process is of a pure jump type and is known in stochastic as compound
Poisson process, see \cite{Skorohod1991}.

Several analytic properties of the jump generator $L$ were studied
recently, see for example \cite{Grigoryan2018,Kondratiev2017,Kondratiev2018}.
Here we recall some of these properties necessary in what follows.

Because $L$ is a convolution operator, it is natural to apply Fourier
transform to study it. 

At first notice that, due to the symmetry of the kernel $a$, its
Fourier image is given by

\[
\hat{a}(k)=\int_{\X}e^{-i(k,y)}a(y)\,\mathrm{d}y=\int_{\X}\cos((k,y))a(y)\,\mathrm{d}y.
\]
Then, it is easy to see that
\[
\hat{a}(0)=1,\quad\hat{a}(k)<1,\;k\neq0,
\]
\[
\hat{a}(k)\to0,k\to\infty.
\]
On the other hand, the Fourier image $L$ is the multiplication operator
by 
\[
\hat{L}(k)=\hat{a}(k)-1
\]
that is the symbol of $L$. 

We make the following assumptions on the kernel $a$.
\begin{description}
\item [{(H)}] The jump kernel $a$ is such that $\hat{a}\in L^{1}(\X)$
and has finite second moment, that is,
\[
\int_{\X}|x|^{2}a(x)\,\mathrm{d}x<\infty.
\]
\end{description}
Denote by $\G_{\la}(x,y)$, $x,y\in\X$, $\la\in(0,\infty)$ the resolvent
kernel of $R_{\la}(L)=(\la-L)^{-1}$. This kernel admits the representation
\[
\G_{\la}(x,y)=\frac{1}{1+\la}\big(\delta(x-y)+G_{\la}(x-y)\big),\quad\la\in(0,\infty),
\]
with 
\begin{equation}
G_{\la}(x)=\sum_{k=1}^{\infty}\frac{a_{k}(x)}{(1+\la)^{k}},\label{eq:Green-kernel}
\end{equation}
where $a_{k}(x)=a^{\ast k}(x)$ is the $k$-fold convolution of the
kernel $a$. Notice that the resolvent kernel $\G_{\la}(x,y)$ has
a singular part, $\delta(x-y)$ and a regular part $G_{\la}(x-y)$.
The Green function, as a generalized function, has the form 
\[
\G_{0}(x)=\delta(x)+G_{0}(x).
\]

The transition probability density $p(t,x)$ in terms of Fourier transform
has representation 
\[
p(t,x)=\frac{1}{(2\pi)^{d}}\int_{\X}e^{i(k,x)+t(\hat{a}(k)-1)}\,\mathrm{d}k.
\]
and for the resolvent kernel 

\[
\G_{\la}(x,y)=-(L-\la)^{-1}(x,y),
\]
holds

\[
\G_{\la}(x-y)=\frac{1}{(2\pi)^{d}}\int_{\X}\frac{e^{i(k,x-y)}}{1-\hat{a}(k)+\la}\,\mathrm{d}k.
\]
For a regularization of the last expression we write 
\[
\frac{1}{1-\hat{a}(k)+\la}=\frac{1}{1+\la}+\frac{\hat{a}(k)}{(1+\la)(1-\hat{a}(k)+\la)}.
\]
Then for operators we have 
\[
\G_{\la}=\frac{1}{1+\la}+G_{\la}
\]
or in terms of kernels 
\[
G_{\la}(x-y)=\frac{1}{(2\pi)^{d}}\int_{\X}e^{i(k,x-y)}\frac{\hat{a}(k)}{(1+\la)(1-\hat{a}(k)+\la)}\,\mathrm{d}k.
\]

We summarize our considerations. The study the resolvent kernel (Green
kernel) is reduced to the analysis of the regular part $G_{0}(x)$,
that is, 
\[
G_{0}(x)=\sum_{k=1}^{\infty}{a_{k}(x)},\qquad a_{k}(x)=a^{\ast k}(x),
\]
where $a_{k}(x)$ is the $k$-fold convolution of the kernel $a$.
The Fourier representation for $G_{0}$ is given by

\[
G_{0}(x)=\frac{1}{(2\pi)^{d}}\int_{\X}e^{i(k,x)}\frac{\hat{a}(k)}{1-\hat{a}(k)}\,\mathrm{d}k.
\]
For $d\geq3$ this integral exists for all $x\in\X$ that follows
from the integrable singularity of $(1-\hat{a}(k))^{-1}$ at $k=0$.
The latter is the consequence of our assumptions on $a(x)$.

\section{Particular Models}

\label{sec:Particular-Models}The main technical question is to obtain
a bound for the $k$-fold convolution $a_{k}(x)$ in $k$ and $x$
together for the analysis of the properties of $G_{0}(x)$. From stochastic
point of view, $a_{k}(x)$ is the density of sum of $k$ i.i.d.~random
variables with distribution density $a(x)$. Unfortunately, we can
not find in the literature any general result in this direction. There
are several particular classes of jump kernels for which we shall
expect such kind of results, see \cite{Kondratiev2018}.
\begin{enumerate}
\item \emph{Exponential tails} or \emph{light tail}s. That is, the kernel
$a(x)$ satisfies the upper bound
\[
a(x)\leq Ce^{-\delta|x|},\quad\delta>0.
\]
\item \emph{Moderate tails}. In this case the asymptotic of $a(x)$ and
$x\to\infty$ is given by
\begin{equation}
a(x)\sim\frac{C}{|x|^{d+\gamma}},\quad\gamma>2.\label{eq:moderate-tails}
\end{equation}
\item Heavy tails. The kernel $a(x)$ has an asymptotic similar to \eqref{eq:moderate-tails}
with $\gamma\in(0,2)$, that is,
\[
a(x)\sim\frac{C}{|x|^{d+\gamma}},\quad\gamma\in(0,2).
\]
\end{enumerate}
In both cases the exponential tails and moderate tails, the kernel
$a(x)$ has second moment. On the other hand, the case of heavy tails
the second moment of $a(x)$ does not exists. 

Below we consider two examples of kernels $a(x)$ and show the bound
for the regular part of the resolvent kernel $G_{0}(x)$. 

\subsection{Gauss Kernels}

Assume that the jump kernels $a(x)$ has the following form:

\begin{equation}
a(x)=C\exp\left(-\frac{b|x|^{2}}{2}\right),\quad C,b>0.\label{G}
\end{equation}

\begin{prop}
If the jump kernel $a(x)$ be given by (\ref{G}) and $d\geq3$, then
holds 
\[
G_{0}(x)\leq C_{1}\exp\left(-\frac{b|x|^{2}}{4}\right).
\]
\end{prop}

\begin{proof}
By a direct calculation we find
\[
a_{k}(x)=\frac{C}{k^{d/2}}\exp\left(-\frac{b|x|^{2}}{2k}\right)
\]
with $C=C(b,d)$. Therefore for $d\geq3$ we obtain
\begin{eqnarray*}
G_{0}(x) & = & \sum_{k=1}^{\infty}{a_{k}(x)}=C\sum_{k=1}^{\infty}\frac{1}{k^{d/2}}\exp\left(-\frac{b|x|^{2}}{2k}\right)\\
 & = & C\sum_{k=1}^{\infty}\sum_{n=0}^{\infty}\frac{1}{k^{d/2}}\frac{1}{n!}\left(-\frac{b|x|^{2}}{2k}\right)^{n}\\
 & = & C\sum_{n=0}^{\infty}\left(\sum_{k=1}^{\infty}\frac{1}{k^{d/2+n}}\right)\frac{1}{n!}\left(-\frac{b|x|^{2}}{2}\right)^{n}.
\end{eqnarray*}
As the series $\sum_{k=1}^{\infty}\frac{1}{k^{d/2+n}}=\zeta(d/2+n)\le\zeta(3/2)$
for $d/2+n>1\Leftrightarrow d\ge3$, where $\zeta(s)$, $s>1$ is
the Riemann zeta function, then we obtain
\[
G_{0}(x)\le C\zeta(3/2)\sum_{n=0}^{\infty}\frac{1}{n!}\left(-\frac{b|x|^{2}}{2}\right)^{n}=C_{1}\exp\left(-\frac{b|x|^{2}}{2}\right).
\]
\end{proof}

\subsection{Exponential Tails}

Now we investigate the case when the jump kernel $a(x)$ has exponential
tails, that is, 
\begin{equation}
a(x)\leq C\exp(-\delta|x|),\quad\delta>0.\label{exp}
\end{equation}

\begin{prop}
If the jump kernel $a(x)$ satisfies (\ref{exp}) and $d\geq3$, then
there exist $A,B>0$ such that the bound for $G_{0}(x)$ holds 
\[
G_{0}(x)\leq A\exp(-B|x|).
\]
\end{prop}

\begin{proof}
It was shown in \cite{Kondratiev2018} that 
\[
a_{n}(x)\leq Cn^{-d/2}\exp(-c\min(|x|,|x|^{2}/n)).
\]
This implies the following bound for $a_{n}(x)$ 
\[
a_{n}(x)\leq Cn^{-d/2}\big(\exp(-c|x|)+\exp(-c|x|^{2}/n)\big).
\]
Hence, it is simple to obtain the bound for $G_{0}(x)$, namely for
$d\ge3$
\begin{eqnarray*}
G_{0}(x) & = & \sum_{n=1}^{\infty}{a_{n}(x)}\le C_{1}\exp(-c_{1}|x|)+C_{2}\exp(-c_{2}|x|^{2})\\
 & \le & A\exp(-B|x|).
\end{eqnarray*}
\end{proof}

\subsection{Brownian Motion}

Let us consider another concrete example of a Markov process. Namely,
denote $B(t)$, $t\ge0$ the Brownian motion in $\X$. The generator
of this process is the Laplace operator $\Delta$ considered in a
proper Banach space $E$. As above we are interested in studying the
expectation of the random variable 
\[
Y(f)=\int_{0}^{\infty}f(B(t))\,\mathrm{d}t
\]
for certain class of functions $f:\X\to\R$. To this end, we introduce
the following class of functions 
\[
CL(\X)=\{f:\X\rightarrow\R:f\;\text{is continuous, bounded and belongs to}\;L_{1}(\X)\}.
\]
It is a Banach space with the norm $\|f\|_{CL}:=\|f\|_{\infty}+\|f\|_{1}$,
where $\|\cdot\|_{\infty}$ is the supremum norm and $\|\cdot\|_{1}$
is the norm in $L_{1}(\X)$.
\begin{prop}
Let $d\geq3$ be given. The Green measure of Brownian motion is 
\[
\G(x,dy)=G_{0}(x-y)\,\mathrm{d}y=\frac{C(d)}{|x-y|^{d-2}}\,\mathrm{d}y.
\]
\end{prop}

\begin{proof}
Note that due to (\ref{FS}) we have 
\[
E^{x}[Y(f)]=-\Delta^{-1}f(x)=\int_{\X}C(d)\frac{f(y)}{|x-y|^{d-2}}\,\mathrm{d}y.
\]
Then 
\begin{align*}
\left|\int_{\X}\frac{f(y)}{|x-y|^{d-2}}\,\mathrm{d}y\right| & \leq\left|\,\int_{|x-y|\leq1}\frac{f(y)}{|x-y|^{d-2}}\,\mathrm{d}y\right|+\left|\,\int_{|x-y|>1}\frac{f(y)}{|x-y|^{d-2}}\,\mathrm{d}y\right|\\
 & \le C_{1}\|f\|_{\infty}+C_{2}\|f\|_{L_{1}(\X)}\\
 & \leq C\|f\|_{CL},
\end{align*}
where we have used the local integrability in $y$ of $|x-y|^{2-d}$.
It means that every function from $CL(\X)$ is integrable with respect
to the Green measure.
\end{proof}
\begin{rem}
In a forthcoming paper we will investigate the additive functionals
for time change Markov processes. More precisely, let $X(t)$, $t\ge0$
be a Markov process in $\X$ with generator $L$ and denote by $\mu_{t}(\mathrm{d}x)$
the marginal distribution of $X(t)$. That is, $\mu_{t}$ is the solution
of the Fokker-Planck equation 
\[
\frac{\partial\mu_{t}}{\partial t}=L^{\ast}\mu_{t}.
\]
In addition, Assume that an inverse subordinator $E(t),$ $t\ge0$
is given and consider random time change 
\[
Y(t):=X(E(t)),\quad t\ge0.
\]
It is known that the marginal distribution $\nu_{t}$ of $Y(t)$ holds
a subordination formula, see \cite{KKdS19}
\[
\nu_{t}(\mathrm{d}x)=\int_{0}^{\infty}D_{t}(\tau)\mu_{\tau}(\mathrm{d}x)\,\mathrm{d}\tau,
\]
where $D_{t}(\tau)$ is the density distribution of $E(t)$. If $P_{t}(x,\mathrm{d}y)$
is the transition probability of $X(t)$, then the Green measure of
$Y(t)$ is (heuristically) given by 
\[
G(x,\mathrm{d}y)=\int_{0}^{\infty}\int_{0}^{\infty}P_{\tau}(x,\mathrm{d}y)G_{t}(\tau)\,\mathrm{d}\tau\,\mathrm{d}t.
\]
But it is not hard to see that this definition leads to a contradiction
and it has to be modified. More precisely, it has to e renormalized
in such a way that then we are able to study the integral functionals
\[
\int_{0}^{\infty}f(Y(t))\,\mathrm{d}t
\]
for a proper class of functions $f:\X\longrightarrow\R$. For the
details, see \cite{KdS20}.
\end{rem}

\subsection*{Acknowledgments}

This work has been partially supported by Center for Research in Mathematics
and Applications (CIMA) related with the Statistics, Stochastic Processes
and Applications (SSPA) group, through the grant UIDB/MAT/04674/2020
of FCT-Funda{\c c\~a}o para a Ci{\^e}ncia e a Tecnologia, Portugal.

\end{document}